\documentclass[11pt]{amsart}

\usepackage[margin=1in]{geometry}
\usepackage{amssymb}
\usepackage[hidelinks]{hyperref}
\usepackage[capitalize]{cleveref}
\usepackage{microtype}
\usepackage{comment}

\usepackage{tikz}

\usepackage{biblatex}
\def\F{\mathcal{F}}

\def\Fq{\mathbb{F}_q}

\def \F{{\mathbb F}}

\crefname{theorem}{Theorem}{Theorems}

\newtheorem{theorem}{Theorem}[section]
\newtheorem{proposition}[theorem]{Proposition}

\newtheorem{corollary}[theorem]{Corollary}
\newtheorem{lemma}[theorem]{Lemma}

\newtheorem{definition}[theorem]{Definition}

\newtheorem*{definition*}{Definition}

\theoremstyle{definition}

\title{Improved bounds for embedding certain configurations in subsets of vector spaces over finite fields}

\author[Bright, Fang, Heritage, Iosevich, and Sun]{Paige Bright, Xinyu Fang, Barrett Heritage, Alex Iosevich, and Maxwell Sun}

\thanks{This work was supported, in part, by the National Science Foundation Grant NSF DMS 2241623 and NSF DMS 1947438. The fourth listed author is supported in part by the National Science Foundation grant NSF DMS 2154232}

\begin{document}

\maketitle

\tableofcontents

\begin{abstract} The fourth listed author and Hans Parshall (\cite{IosevichParshall}) proved that if $E \subset {\mathbb F}_q^d$, $d \ge 2$, and $G$ is a connected graph on $k+1$ vertices such that the largest degree of any vertex is $m$, then if $|E| \ge C q^{m+\frac{d-1}{2}}$, for any $t>0$, there exist $k+1$ points $x^1, \dots, x^{k+1}$ in $E$ such that $||x^i-x^j||=t$ if the $i$'th vertex is connected to the $j$'th vertex by an edge in $G$. In this paper, we give several indications that the maximum degree is not always the right notion of complexity and prove several concrete results to obtain better exponents than the Iosevich-Parshall result affords. This can be viewed as a step towards understanding the right notion of complexity for graph embeddings in subsets of vector spaces over finite fields. \end{abstract}

\section{Introduction}

The Erdos distance problem is one of the fundamental questions in geometric combinatorics. In its original formulation, the question asks, what is the smallest number of distinct distances determined by a subset of ${\mathbb R}^d$, $d \ge 2$, of a given size? Erd\H os conjectured in 1946 (\cite{Erdos1946}) that if $P$ is a finite subset of ${\mathbb R}^d$, $d \ge 2$ of size $N$, then for any $\epsilon>0$ there exists $C_{\epsilon}>0$ such that 
$$ \# \Delta(P) \ge C_{\epsilon} N^{\frac{2}{d}-\epsilon},$$ where 
$$ \Delta(P)=\{|x-y|: x,y \in P\},$$ with $|x|=\sqrt{x_1^2+x_2^2+\dots+x_d^2}$, the usual Euclidean distance. After more than half a century of efforts by many outstanding mathematicians, the conjecture was resolved in two dimensions by Larry Guth and Nets Katz (\cite{GK15}). In higher dimensions, the conjecture is still open, with the best-known results due to Josef Solymosi and Van Vu (\cite{SV08}). 

The purpose of this problem is to study point configuration problems stemming from Erdos type distance problems in vector spaces over finite fields. Throughout this paper, we let $q$ denote an odd prime power. For $d\geq 2$, $\mathbb{F}_q^d$ is the $d$-dimensional vector space over the field $\mathbb{F}_q$ with $q$ elements. Given $x = (x_1,\dots, x_d) \in \mathbb{F}_q^d$, we let 
\[
\lVert x\rVert = x_1^2 + \dots+ x_d^2.
\]

In finite fields, the problem takes on its own unique characteristics. For example, suppose that $q$ is a prime congruent to $1$ modulo $4$. Then $i=\sqrt{-1}$ is in ${\mathbb F}_q$, which leads to the following example that cannot occur in Euclidean space. Let 
$$ E=\{(t,it): t \in {\mathbb F}_q \}.$$ It follows that if we define the distance set 
$$ \Delta(E)=\{||x-y||: x,y \in E\},$$ then $\Delta(E)=\{0\}$. In order to get around this issue, Jean Bourgain, Nets Katz and Terry Tao, who proved the initial result for distance sets in vector spaces over finite fields (\cite{bourgain2006sumproduct}), assumed that $q$ is a prime congruent to $3$ modulo $4$. Another way around this difficulty was devised by the fourth listed author of this paper and Misha Rudnev. They proved in \cite{IR07} that if $E \subset {\mathbb F}_q^d$, $d \ge 2$, $t \not=0$, then 
$$ |\{(x,y) \in E \times E: ||x-y||=t \}|={|E|}^2q^{-1}+error,$$ where 
$$ |error| \leq 2q^{\frac{d-1}{2}}|E|.$$

The ``error" is smaller than the main term when $|E|>2q^{\frac{d+1}{2}}$, and in such a case, all non-zero distances are realized. Since the zero distance is realized by taking $x=y$, all distances are in fact realized. This result can be viewed as a statement about two point configurations or, more precisely, about the embedding of complete graphs on two vertices, in this sense, in subsets of ${\mathbb F}_q^d$ of a given size. This raises the question about more complicated point configurations, and graphs that determine, as well as their embeddings, in the sense to be made precise below, in subsets of ${\mathbb F}_q^d$. 

The question we particularly focus on is the following: how large does a subset of $E\subseteq\F_q^d$ need to be such that $E$ contains a specified graph of points in $\F_q^d$ with distances assigned between edges? In particular, we are interested in embedding \textit{distance graphs} into finite field vector spaces. We call a graph $\mathcal G = (V,E)$ a \textit{distance graph} when for each edges $e\in E$, there is some associated nonzero edge length $\lambda_e\in \F_q^\ast$. Then, we call $X$ an \textit{isometric copy} of $\mathcal G$ when there exists a distance preserving bijection $\varphi: V\to X$ where for each $v,w\in V$ with an edge $e$ connecting $v$ and $w$, we have $|\varphi(v) - \varphi(w)|^2 = \lambda_e$. We are interested in estimating the following values for various graphs $\mathcal G$:
\begin{align*}
\mathcal N_{\mathcal G}(E) &:= \#\{\text{embeddings of a distance graph $\mathcal G$ in $E$}\}, \\
\mathcal N_{\mathcal G}^\ast(E) &:= \#\{\text{non-degenerate embeddings of a distance graph $\mathcal G$ in $E$}\} \\
\mathcal D_{\mathcal G}(E) &:= \#\{\text{degenerate embeddings of a distance graph $\mathcal G$ in $E$}\}
\end{align*}

In recent work by Alex Iosevich and Hans Parshall (see \cite{IosevichParshall}), they showed that given $n,s\in \mathbb{N}$ and $E\subseteq\F_q^d$ with $|E|\geq 12 n^2 q^{\frac{d-1}{2} + s}$, $E$ contains an isometric copy of every distance graph with $n$ vertices and maximum vertex degree $s$. This result does not bode well for graphs with high maximum vertex degree, which begs the question: given distance graphs with relatively high vertex degree, how large does a subset $E\subseteq\F_q^d$ need to be to contain that distance graph? In this paper, we study specific configurations of points to improve upon Iosevich and Parshall's results in certain cases.

The simplest example of a graph having high vertex degree but low complexity is a star with many legs. A $k$-star is a special case of a tree of size $k$, which can be embedded easily as shown in the following theorem by Iosevich, Jardine and McDonald. 


\begin{theorem}[\cite{cycles}]
Let $T$ be a tree with $k+1$ vertices and hence $k$ edges. For $\epsilon>0$, if $|E|>q^{\frac{d+1}{2}+\epsilon}$, then we have
\[
\mathcal{N}_T (E)\geq \frac{|E|^{k+1}}{q^k}\left(1- 8 q^{-\frac{2\epsilon}{k+1}}\right)
\]
where $\mathcal{N}_T (E)$ is the number of embeddings of $T$ into $E$.

\end{theorem}

Notice that the size of the tree only affects the constant but not the exponent for the size of $E$ required to embed it. Therefore, this result implies that we can embed a star of any degree in a subset of size asymptotically $q^{\frac{d+1}{2}}$, with a constant that grows with the degree. This greatly improves the bound given by Iosevich-Parshall, which becomes trivial whenever the  degree is greater than or equal to $\frac{d+1}{2}$. 

\begin{center}
\begin{figure}[h]
\label{fig:star}
  \centering
  \begin{tikzpicture}
    \coordinate (center) at (0,0);
    \fill (center) circle (2pt);

    \foreach \i in {1,...,6} {
      \coordinate (v\i) at (\i*360/6:2);
      \fill (v\i) circle (2pt);
      \draw (center) -- (v\i);
    }
  \end{tikzpicture}
  \caption{a $6$-star}
\end{figure}
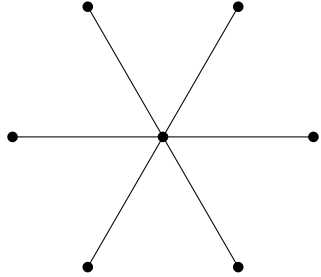
\end{center}

In this paper, we only consider distance graphs whose edge lengths are identical, and denote that by $t\in\F_q^*$. From now on, whenever we refer to an edge in a distance graph, it is assumed to have length $t$.

\subsection{H\"older Extensions}

We first define the H\"older extension of a graph.

\begin{definition}[H\"older extension of a Graph] 
    Let $G$ be a simple graph with vertex set $V$. Let $e$ be its edge function defined by $e(u,v)=1$ if $u,v\in V$ are connected and $0$ otherwise. 
Let $S=\{v_1,\cdots, v_n\}$ be a subset of $V$. Let $m\geq 2$ be an integer.
Then the \emph{$k$-H\"older extension} of this graph (with respect to the set $S$) is a new graph $G'$ defined as follows. The vertex set of $G'$ consists of $V$ plus $k-1$ copies of the vertices in $S$, labelled $\{v_j^i: 1\leq j\leq n, 1\leq i\leq k-1\}$. The edge function $e'$ of $\mathcal G'$ is defined as follows.
\[
e'(v_j^i, w) = \begin{cases}
    e(v_j,w), & w\in V\backslash S\\
    e(v_j,v_{j'}), & w=v^i_{j'}\\
    0, & \text{otherwise}
\end{cases}
\]
and $e'=e$ on $V \times V$.
\end{definition}

We call this a H\"older extension because it is relatively easy to apply H\"older's inequality in order to lower bound the number of embeddings of these graphs. We have the following results.

\begin{lemma}
\label{holderLem}
    Let $E\subseteq\Fq^d, d\geq 2$. Let $G$ be a H\"older extension of $H$ of degree $k\geq 2$, with $S=\{v_1,\cdots, v_n\}$, $V\backslash S = \{w_1,\cdots,w_l\}$, where $V$ is the vertex set of $H$ and $S$ is the subset of vertices being duplicated in $G$. Then we have that 
    \[
    \mathcal N_G(E) \geq \frac{(\mathcal N_H(E))^k}{(\mathcal N_{H \setminus S}(E))^{(k-1)l}}.
    \]
\end{lemma}
\begin{corollary}
\label{holderLemCor}
    Let $E\subseteq\Fq^d, d\geq 2$. Let $G$ be a H\"older extension of $H$ of degree $k\geq 1$, with $S=\{v_1,\cdots, v_n\}$, $V\backslash S = \{w_1,\cdots,w_l\}$. Then we have
    \[
    \mathcal N_G(E)
    \geq 
    \frac{(\mathcal N_H(E))^k}{|E|^{(k-1)l}}.
    \]
\end{corollary}

\Cref{holderLemCor} follows immediately from Lemma \ref{holderLem} since the number of embeddings of $H\setminus S$ in $E$ is at most $|E|^l$. \Cref{sec: prelim} gives a proof of \Cref{holderLem} and also states more results regarding counting degenerate embeddings of H\"older extensions.

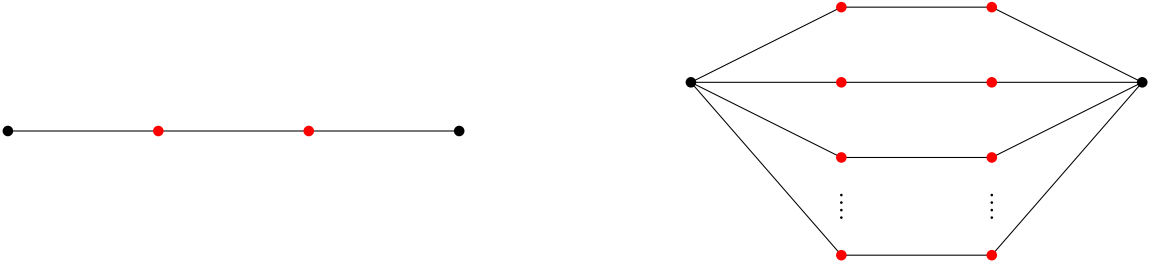
\begin{figure}[h]
\begin{minipage}{0.45\textwidth}
\centering
\begin{tikzpicture}[scale = 1.0]
  \fill (0,0) circle (2pt);
  \fill (6,0) circle (2pt);

  \draw (0,0) -- (2,0) -- (4,0) -- (6,0);
  \fill[color=red] (2,0) circle (2pt);
  \fill[color=red] (4,0) circle (2pt);
  
\end{tikzpicture}
\end{minipage}\hfill
\begin{minipage}{0.45\textwidth}
\centering
\begin{tikzpicture}[scale = 1.0]
  \fill (0,0) circle (2pt);
  \fill (6,0) circle (2pt);
  
  \draw (0,0) -- (2,1) -- (4,1) -- (6,0);
  \fill[color=red]  (2,1) circle (2pt);
  \fill[color=red]  (4,1) circle (2pt);
  
  \draw (0,0) -- (2,0) -- (4,0) -- (6,0);
  \fill[color=red]  (2,0) circle (2pt);
  \fill[color=red]  (4,0) circle (2pt);
  
  \draw (0,0) -- (2,-1) -- (4,-1) -- (6,0);
  \fill[color=red]  (2,-1) circle (2pt);
  \fill[color=red]  (4,-1) circle (2pt);

  \foreach \y in {-1.5, -1.6, -1.7, -1.8}
    \fill (2, \y) circle (0.5pt);
  \foreach \y in {-1.5, -1.6, -1.7, -1.8}
    \fill (4, \y) circle (0.5pt);

  \draw (0,0) -- (2,-2.3) -- (4,-2.3) -- (6,0);
  \fill[color=red]  (2,-2.3) circle (2pt);
  \fill[color=red]  (4,-2.3) circle (2pt);
  
\end{tikzpicture}
\end{minipage}
\caption{a 3-chain and the $k$-H\"older extension of it with respect to the middle vertices}
\label{fig:holderChain}
\end{figure}

\subsection{H\"older Extensions of Chains}
The first example we study are H\"older extensions of chains. 

\begin{definition}[$m$-chain]
An \emph{$m$-chain} is a distance graph on $m+1$ vertices $(x_1,\dots, x_{m+1})$ such that $x_i$ and $x_{i+1}$ are connected by an edge for $1\leq i \leq m$. 
\end{definition}
We denote an $m$-chain by $C_m$. The number of non-degenerate embeddings of $m$-chains in a subset $E\subseteq\F_q^d$ were studied by Bennett et.al. in \cite{LongPaths}. 

In \Cref{sec:holderChains}, we consider the $k$-H\"older extension of an $m$-chain fixing the endpoints, denoted by $G_{k,m}$, as illustrated in \Cref{fig:holderChain}. \Cref{holderLemCor} immediately implies the following given input on the count of embeddings of $m$-chains.
\begin{theorem}
    Given that $|E| > \frac{4k}{\ln 2}q^{\frac{d+1}{2}}$, 
    \[
    \mathcal{N}_{G_{k,m}}(E) \ge \frac{|E|^{km-k+2}}{q^{km}}\left(1-\frac{4mq^{\frac{d+1}{2}}}{\ln 2 \cdot |E|}\right)^k.
    \]
\end{theorem}
More importantly, we give further results on the number of \emph{non-degenerate} embeddings of such configurations in \Cref{sec:holderChains}.

\subsection{Chains of Simplices}
The second example we study are chains of $m$-simplices, discussed in detail in \Cref{sec:chainsOfSimplices}. We first recall the definition of an $m$-simplex.
\begin{definition}[Regular $m$-simplex] An $m$-simplex is a distance graph on $m+1$ vertices such that there is an edge 
between every pair of vertices.   
\end{definition}
Now, we can form chains of simplices by linking copies of them up in a certain way. For example, \Cref{fig:chainSimplex} shows a chain of $2$-simplices.

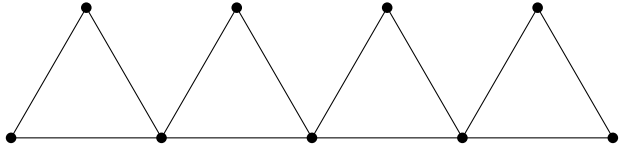
\begin{figure}[h]
\centering
\begin{tikzpicture}[scale = 2]
  \foreach \x in {0, 1, 2, 3}
    \draw (\x,0) -- (\x+0.5,{0.5*sqrt(3)}) -- (\x+1,0) -- cycle;
    
  \foreach \x in {0, 1, 2, 3}
    \fill (\x,0) circle (1pt);
  \foreach \x in {0, 1, 2, 3}
    \fill (\x+0.5,{0.5*sqrt(3)}) circle (1pt);
  \fill (4,0) circle (1pt);
    
\end{tikzpicture}
\caption{4-chain of 2-simplices}
\label{fig:chainSimplex}
\end{figure}

We define a $k$-chain of $m$-simplices precisely as follows.
\begin{definition}[$k$-chains of $m$-simplices]
    Let $k\geq 3$. A \emph{$k$-chain of $m$-simplices}
    is a distance graph obtained by replacing each edge in a $k$-chain with an $m$-simplex, with two vertices in the simplex at the two endpoints of the original edge of the chain.
\end{definition}
Let $T^m_k$ denote a $k$-chain of $m$-simplices. We have the following results regarding embeddings of $T_k^m$ into subsets of $\F_q^d$.
\begin{theorem}
\label{thm:kChainsTotal}
    Let $\ell$ be such that $2^\ell < k \leq 2^{\ell+1}$. Suppose $E\subset\F_q^d$ with $m<\frac{d+1}{2}$.
    Whenever $|E| > 2^{\ell+4}\cdot 3(m+1)^2 q^{\frac{d-1}{2}+m}$, 
    \[
    \mathcal{N}_{T^m_{k}}(E) \ge \frac{|E|^{mk+1}}{q^{k\binom{m+1}{2}}}\cdot 2^{2k - 2^{\ell+1}(m+2) - 2}.
    \]    
\end{theorem}
\begin{theorem}
\label{thm:kChainsNondeg}
    Let $\ell$ be such that $2^\ell < k \leq 2^{\ell+1}$. Suppose $E\subset\F_q^d$ with $m<\frac{d+1}{2}$.
    Whenever $|E| > 2^{k+2^{\ell+1}(m+2)+3}\cdot 3(m+1)^2 q^{\frac{d-1}{2}+m}$,
    \[
    \mathcal{N}^*_{T^m_{k}}(E) \ge \frac{|E|^{mk+1}}{q^{k\binom{m+1}{2}}}\cdot 2^{2k - 2^{\ell+1}(m+2) - 3}.
    \]
\end{theorem}

\subsection{Trees of Simplices}

Lastly, in \Cref{sec:TreesOfSimplices}, we study trees of $m$-simplices. These are a natural generalization of chains of $m$-simplices, defined as follows.

\begin{definition}[Tree of $m$-simplices]
    A \emph{tree of $m$-simplices} is the graph obtained from a usual tree by replacing each of its edges with an $m$-simplex, with two vertices of the simplex being the two endpoints of the original edge. We will call the vertices coming from the original tree \emph{nodes}, and the $m$-simplices \emph{connections}. The \emph{size} of such a tree is the number of simplices it contains.
\end{definition}
\begin{theorem}
\label{thm:TreesOfSimplicesTotal}
Let $T_\ell$ be a tree of $m$-simplices of size $\ell$. Suppose 
\[|E|\geq 12(m+1)^2\cdot 2^{(\ell-1)(m+5)}\cdot q^{m+\frac{d-1}{2}}.\] 
Then
\[
\mathcal{N}_{T_{\ell}}(E)\geq 2^{-\eta(m,\ell)}\frac{|E|^{m\ell+1}}{q^{\binom{m+1}{2}\cdot\ell}}
\]
where \[\eta(m,\ell)=(\ell-1)(m+5)(\ell m/2+1)-(m+2)\ell.\]
\end{theorem}

\begin{theorem}
\label{thm:TreesOfSimplicesNondeg}
    Let $T_\ell$ be a tree of $m$-simplices of size $\ell$. Suppose 
\[|E|\geq 2^{\eta(m,\ell)+3\ell}\cdot 12(m+1)^2\cdot q^{m+\frac{d-1}{2}}.\] 
Then
\[
    \mathcal{N}^*_{T_{\ell}}(E) \geq 2^{-\eta(m,\ell)+3\ell+1}\frac{|E|^{m\ell+1}}{q^{\binom{m+1}{2}\cdot\ell}}
\]
where $\eta$ is as defined in \Cref{thm:TreesOfSimplicesTotal}.
\end{theorem}

The constants grow/diminish exponentially with $\ell$ and $m$, but if we are willing to keep them constant, then we still get the statistically correct lower bounds. 

\section{Preliminary Tools}\label{sec: prelim}
We first prove \Cref{holderLem}, which is a useful tool for obtaining a lower bound on the number of embeddings of a graph that is a H\"older extension of a simpler graph. 

\begin{proof}[Proof of Lemma \ref{holderLem}]
    Let $F(x_1,\cdots, x_l)=F(x_I)$ be the indicator function for the configuration $H-S$; i.e. $F(x_I)=1$ if $(w_1 = x_1,\cdots, w_l = x_l)$ is an embedding of $H-S$ and $F(x_I)=0$ otherwise. Let $G(x_1,\cdots,x_l, y_1,\cdots, y_k)=G(x_I, y_J)$ be the indicator function for the configuration $G$ defined similarly, where $w_i=x_i$ and $v_j=y_j$.
    
    Then 
    \[
    \mathcal N_H(E)
    =
    \sum_{x_1,\cdots,x_l\in\Fq^d}\left(
    F(x_I)
    \sum_{y_1,\cdots,y_n\in\Fq^d} G(x_I, y_J)\right).
    \]
    Therefore,
    \begin{align*}
    \mathcal N_G(E)
    &=
    \sum_{x_1,\cdots,x_l\in\Fq^d}\left(
    F(x_I)
    \sum_{y_1,\cdots,y_n\in\Fq^d} G(x_I, y_J)\right)^m\\\tag{H\"older's inequality}
    &\geq \frac{\left(\sum_{x_1,\cdots,x_l\in\Fq^d}
    F(x_I)
    \sum_{y_1,\cdots,y_n\in\Fq^d} G(x_I, y_J)\right)^k}{\left(\sum_{x_1,\cdots,x_l\in\Fq^d}
    F(x_I)\right)^{k-1}}\\
    &=
    \frac{(\mathcal N_H(E))^k}{(\mathcal N_{H\setminus S}(E))^{k-1}}.    
    \end{align*}
\end{proof}

Below, we give a means for counting degeneracies for embeddings of connected graphs. This is helpful since \Cref{holderLem} gives a lower bound on the number of total embeddings but allows for possibly many degenerate counts.

\begin{theorem} \label{counting_degens_general}
Fix a subset $L$ of $\F_q$. For $\epsilon > 0$ and a positive integer $n \ge 2$, if $|E| > q^{\frac{d+1}{2} + \epsilon}$, there exists a subset $E' \subseteq E$ such that
\begin{equation*}
|E\setminus E'| \le 2|L|q^{-\frac{2\epsilon}{n}}|E|
\end{equation*}
and the total number of embeddings of all connected distance graphs with $n$ vertices with edge lengths all chosen from $L$ in $E'$ is bounded above as follows:
\begin{equation} \label{degen_upper_bound}
    \mathcal{N}_{n,L}(E') \le c_{n,{|L|}} \frac{|E|^{n}}{q^{n-1}}
\end{equation}
where $c_{n,{|L|}} = 2^{\binom{n}{2}-n+2} |L|^{n-1} n^{n-2}$.

In addition, if $G$ is a connected graph with edge lengths chosen solely from $L$ with $n+1$ vertices, then the number of degenerate embeddings of $G$ in $E'$ is bounded above by 
\[
n(n+1)\frac{|E|^{n}}{q^{n-1}}.
\]
\end{theorem}
\begin{proof}
    Let
\begin{equation*}
E_t = \left\{ x \in E: \sum_{y \in E} E(y)\phi_t(x,y) \le \lambda \frac{|E|}{q} \right\}, \quad E' = \bigcap_{t \in L} E_t.
\end{equation*}
We have,
\begin{equation}
|E \setminus E_t| \le \frac{q}{\lambda |E|}\sum_{x,y} E(x)E(y)\phi_t(x,y) \le 2\frac{|E|}{\lambda}
\end{equation}
by \cite{LongPaths}, so
\begin{equation}
|E \setminus E'| \le 2|L|\frac{|E|}{\lambda}.
\end{equation}
We need the following result. It is analogous to Lemma 4.4 in \cite{cycles} and proved in exactly the same way.
\begin{lemma}
Let $T$ be a tree with $r+1$ vertices and $r$ edges with edge lengths only from the set $L$. If $n_T^*$ is the number of embeddings of $T$ in $E'$, then
\[
\left| n_T^* - \frac{|E|^{r+1}}{q^r} \right| \le 4r \left( \lambda^{-1} + \lambda^{\frac{r-1}{2}} \frac{q^{\frac{d+1}{2}}}{|E|} \right) \frac{|E|^{r+1}}{q^r}.
\]
\end{lemma}

We can set $\lambda = q^{2\epsilon/(r+1)}$ and use the above lemma to obtain
\begin{equation} \label{tree_imbedding}
\left| n_T^* - \frac{|E|^{r+1}}{q^r} \right| \le 8\cdot \frac{|E|^{r+1}}{q^r}\cdot q^{-\frac{2\epsilon}{r+1}}.
\end{equation}
Now, note that all graphs with $n$ vertices are spanned by a tree with at most $n$ vertices. Given an assignment of lengths in $L$ to edges, the number of embeddings of a graph is at most the number of embeddings of any spanning tree with analogous edge assignment. Let $G_{T}$ be the number of simple graphs with $n$ vertices with $T$ as a spanning tree. For a given $T$, we have at most $|L|^{n-1}$ ways to assign the edge lengths. Thus, summing over all trees $T$ with $n$ vertices, we have,
\begin{align}
    \mathcal{G}_{n}(E') &\le \sum_{T} G_T |L|^{n-1} n_T^* \\
    &= 2^{\binom{n}{2}-n+1}|L|^{n-1} \cdot n^{n-2} \cdot 2\frac{|E|^n}{q^{n-1}} \\
    &\le 2^{\binom{n}{2}-n+2} |L|^{n-1}\cdot  n^{n-2} \cdot \frac{|E|^n}{q^{n-1}}.
\end{align}
The last part of the theorem follows from \cref{tree_imbedding} and noting that a degeneration involves setting two vertices equal, which results in a connected graph with $n-1$ vertices.
\end{proof}

We introduce the following ``Shaving Lemma", which allows us to find a subset of a given set $E\subseteq\F_q^d$ in which the vertex degrees are close to the statistically correct count, without losing too many points. 
Let 
\[
S(x):=\begin{cases}
    1, & ||x||=t\\
    0, &\text{ otherwise.}
\end{cases}
\]
Then we have that the convolution
\[
E*S (x) = \sum_{y\in\F_q^d} E(y)S(x-y) = \#\{\text{edges in $E$ incident on $x$}\}.
\]
\begin{theorem}[Shaving Lemma for One Edge]
    Let $E\subseteq\F_q^d$ with $|E|\geq Cq^{\frac{d+1}{2}}$, where $C>4$. Then for any $\lambda>\frac{2C}{C-4}$,  there exists a subset $E^*\subseteq E$ with 
    \begin{equation*}
        |E^*|\geq \frac{(1 - 2\lambda^{-1})}{2\lambda}|E|
    \end{equation*}
    such that 
    \begin{equation*}
        \lambda^{-1} \frac{|E|}{q} \leq E*S(x)\leq \lambda\frac{|E|}{q}.
    \end{equation*}
\end{theorem}
\begin{proof}
    Let 
    \[
    E' := \left\{x\in E: E*S(x)\leq \lambda \cdot \frac{|E|}{q}\right\}.
    \]
    By counting the number of edges in $E$ incident to vertices in $E\backslash E'$, we see that 
    \[
    |E\backslash E'|\cdot \lambda \frac{|E|}{q}\leq \sum_{x}E*S(x)\leq \frac{2|E|^2}{q},
    \]
    where the last inequality holds since $|E|\geq 2q^{\frac{d+1}{2}}$.
    Thus,
    $
    |E\backslash E'|\leq\frac{2|E|}{\lambda}
    $,
    which implies
    \[
    |E'|\geq |E|\left(1-2\lambda^{-1}\right)\geq 4q^{\frac{d+1}{2}}
    \]
    by our choice of $\lambda$.
    Now let 
    \[
    E^*:=\left\{x\in E': E'*S(x)\geq \lambda^{-1}\cdot\frac{|E'|}{q}\right\}.
    \]
    Then,
\[
|E'\setminus E^*| \cdot\lambda^{-1}\frac{|E'|}{q} + |E^*| \cdot\lambda \frac{|E|}{q}
\geq  \sum_{x} E'*S(x) \geq \frac{|E'|^2}{2q}.
\] 
Upper bounding $|E'\setminus E^*|$ by $|E'|$ and rearranging, we get
\[
\implies |E^*| \ge \left(\frac{1}{2}-\lambda^{-1}\right)\frac{\frac{|E'|^2}{q}}{\lambda\frac{|E|}{q}} = \frac{1}{2\lambda}(1 - 2\lambda^{-1})|E|.
\]
    
\end{proof}
Note in the above proof that the constants $\lambda$ and $\lambda^{-1}$ in the upper and lower bounds can be made independent, as long as they are chosen in an appropriate range.
In fact, this theorem can be stated much more generally for an arbitrary configuration with a given basepoint for which we have a lower bound on the size of a subset containing the statistically correct number of embeddings. 
The proof follows essentially the same idea as above.

\begin{theorem}[General Shaving Lemma]
    Let $\mathcal{G}$ be a distance graph with $n$ vertices and $m$ edges, and a chosen vertex $v$ (which we call the ``basepoint"). Suppose that, for all $E\subseteq\F_q^d$ with $|E| > N$,
    \[
    \frac{|E|^{n}}{2q^m} \le \mathcal{N}_\mathcal{G}(E) \le \frac{2|E|^{n}}{q^m}.
    \]
    For $x \in E$ let
    \begin{equation}
        f(x):=\#\{\text{embeddings of $\mathcal{G}$ in $E$ with $v$ at $x$}\}.
    \end{equation}
    For $\lambda_2 \ge 4$ and $\lambda_1 < 2^{-n}$, we have that, given $|E| > 2N$,
    \begin{equation}
        \left| \left\{x \in E: \lambda_1\frac{|E|^{n-1}}{q^m} \le f(x) \le \lambda_2\frac{|E|^{n-1}}{q^m} \right\}\right| \ge \left( \frac{1}{2} - 2^{n-1}\lambda_1 \right)(1 - 2\lambda_2^{-1})^n|E|.
    \end{equation}
\end{theorem}
\begin{proof}
    Let
    \[
    E'=\left\{x\in E: f(x)\leq \lambda_2 \cdot \frac{|E|^{n-1}}{q^m}\right\},
    \]
    We have,
    \[
    |E \setminus E'| \le \frac{q^m}{\lambda_2 |E|^{n-1}}\sum_{x \in E} f(x) \le \frac{q^m}{\lambda_2 |E|^{n-1}} \frac{2|E|^n}{q^m} = \frac{2}{\lambda_2}|E|
    \]
    which implies
    \begin{equation}
        |E'| \ge (1 - 2\lambda_2^{-1})|E| \ge \frac{|E|}{2}.
    \end{equation}
    Now, let
    \begin{equation}
        g(x):=\#\{\text{embeddings of $\mathcal{G}$ in $E'$ with $v$ at $x$}\}.
    \end{equation}
    and
    \[
    E^* = \left\{x \in E': g(x) \ge \lambda_0\frac{|E'|^{n-1}}{q^m}\right\}.
    \] 
    We obtain that
    \[
    \frac{|E'|^n}{2q^m} \le \mathcal{N}_\mathcal{G}(E') = \sum_{x \in E'} g(x) \le |E' \setminus E^*|\lambda_0 \frac{|E'|^{n-1}}{q^m} + |E^*|\lambda_2\frac{|E|^{n-1}}{q^m}
    \]
    \[
    \implies |E^*| \ge \left( \frac{1}{2} - \lambda_0 \right) \frac{|E'|^n}{q^m} \Big/ \left[ \lambda_2\frac{|E|^{n-1}}{q^m} \right] \ge \left( \frac{1}{2} - \lambda_0 \right)(1 - 2\lambda_2^{-1})^n|E|.
    \]
    For any $x \in E^*$, we obtain that
    \[
    f(x) \ge g(x) \ge \lambda_02^{1-n} \frac{|E|^{n-1}}{q^m}.
    \]
    Now, let $\lambda_1 = 2^{1-n}\lambda_0$ to obtain the result.
\end{proof}

\section{H\"older Extension of $m$-Chains}
\label{sec:holderChains}
In this section, we study H\"older extensions of $m$-chains. In particular, we find a bound on the total number of $k$-fold H\"older extensions of an $m$-chain given a restriction on $|E|$, then deduce a count for non-degenerate extensions. We do this by counting degeneracies.

Assume $k \geq 3$, since $k=2$ implies a cycle. 
Recall that $C_m$ denotes an $m$-chain and $G_{k,m}$ denotes the $k$-fold H\"older extension of $C_m$ with two fixed endpoints.

We need the following input by Bennett et al. on the number of non-degenerate embeddings of chains.
\begin{theorem}[\cite{LongPaths}]
    Suppose $|E|\geq \frac{4m}{\ln 2}\cdot q^{\frac{d+1}{2}}$. Then 
    \[
    \left|\mathcal{N}^*_{C_m}(E)-\frac{|E|^{m+1}}{q^m}\right|
    \leq \frac{4m}{\ln 2}\cdot q^{\frac{d+1}{2}}\cdot \frac{|E|^m}{q^m}.
    \]
\end{theorem}
We immediately deduce the following.
\begin{theorem}
    Given that $|E| > \frac{4k}{\ln 2}q^{\frac{d+1}{2}}$, 
    \[
    \mathcal{N}_{G_{k,m}}(E) \ge \frac{|E|^{km-k+2}}{q^{km}}\left(1-\frac{4mq^{\frac{d+1}{2}}}{\ln 2 \cdot |E|}\right)^k.
    \]
\end{theorem}
\begin{proof}
By Corollary \ref{holderLemCor},
\[
\mathcal{N}_{G_{k,m}}(E)\geq \frac{\mathcal{N}^*_{G_{1,m}}(E)^k}{|E|^{2(k-1)}}
\geq 
\frac{|E|^{k(m-1)+2}}{q^{km}}\left(1-\frac{4mq^{\frac{d+1}{2}}}{\ln 2 \cdot |E|}\right)^k.
\]
\end{proof}

\begin{theorem}
    When $|E| > \max\left((km)^{km}q^{\frac{d+1}{2}}, 2(km)^2q^k\right)$,
    \[
    \mathcal{N}_{G_{k,m}}^*(E) \ge \frac{|E|^{k(m-1)+2}}{2q^{km}} \left(1 - \frac{1}{km} - \frac{2(km)^2q^k}{|E|}\right),
    \]
    which is positive when
    \[
    |E| \ge \frac{2(km)^3q^k}{km-1}.
    \]
    Alternatively, assume that $|E| > (\log q) \max\left( q^{\frac{d+1}{2}}, q^k \right)$. Then we obtain an asymptotic lower bound for sufficiently large $q$:
    \[
    \frac{|E|^{k(m-1)+2}}{q^{km}}\left(1-\frac{8mk}{\ln 2\ln q} 
- 2km (\ln q)^{-\frac{2}{km}} \right).
    \]
\end{theorem}
\begin{proof}
We first take the general assumption that $q \ge f(q)q^{\frac{d+1}{2}}$. By \Cref{counting_degens_general}, there exists a subset $E' \subseteq E$ such that
\begin{equation*}
|E\setminus E'| \le \lambda|E| \text{~~in other words~~} (1-\lambda) |E| \leq |E'|.
\end{equation*} 
where $\lambda=2f(q)^{-\frac{2}{k(m-1)+1}}$,
so that 
\begin{equation} \label{degen_upper_bound_general}
    \mathcal{D}_{G_{m,k}}(E')\leq c_{k,m}\frac{|E|^{k(m-1)+1}}{q^{k(m-1)}},
\end{equation}
where $c_{k,m}=(k(m-1)+2)(k(m-1)+1)$. Now it suffices to choose a suitable $f(q)$ so that when $|E|>f(q)q^{\frac{d+1}{2}}$ we have that
\[
|E'| > \frac{8km}{\ln 2}q^{\frac{d+1}{2}}
\]
so that
\begin{align}
\mathcal{N}_{G_{k,m}}(E')
&\geq 
\frac{|E'|^{k(m-1)+2}}{q^{km}}\left(1-\frac{4mq^{\frac{d+1}{2}}}{\ln 2 \cdot |E'|}\right)^k\nonumber\\
&\geq
\frac{|E|^{k(m-1)+2}}{q^{km}}\left(1-\frac{4mkq^{\frac{d+1}{2}}}{\ln 2 \cdot |E'|}\right)
\left(1-\lambda\right)^{k(m-1)+2} \nonumber\\
&\geq
\frac{|E|^{k(m-1)+2}}{q^{km}}\left(1-\frac{4mk}{\ln 2(1-\lambda)f(q)}\right) (1 - (k(m-1) + 2)\lambda). \label{lower_bound_count}
\end{align}

Since $|E'|\geq (1-\lambda)|E|$, it suffices to have
\[
f(q) (1-\lambda) = f(q)\left(1 - f(q)^{-\frac{2}{k(m-1)+1}}\right) > \frac{8km}{\ln 2}.
\]
This is clearly true by the assumptions. We can then simplify \eqref{lower_bound_count} to obtain
\[
\mathcal{N}_{G_{k,m}}(E') \ge \frac{|E|^{k(m-1)+2}}{2q^{km}} \left(1-km f(q)^{-\frac{2}{km}}\right) \ge \frac{|E|^{k(m-1)+2}}{2q^{km}} \left(1 - \frac{1}{km}\right).
\]
From (\ref{degen_upper_bound_general}), we can get
\[
\mathcal{D}_{G_{m,k}}(E') < (km)^2 \frac{q^k}{|E|} \cdot \frac{|E|^{k(m-1)+2}}{q^{km}}
\]
so that 
\[
\mathcal{N}_{G_{k,m}}^*(E) \ge \frac{|E|^{k(m-1)+2}}{2q^{km}} \left(1 - \frac{1}{km} - \frac{2(km)^2q^k}{|E|}\right)
\]
Now, if we alternatively assume that $f(q) = \log q$, then $\lambda < \frac{1}{2}$ for sufficiently large $q$ and (\ref{degen_upper_bound_general}) turns into
\begin{align}
    \mathcal{N}_{G_{k,m}}(E') &\ge \frac{|E|^{k(m-1)+2}}{q^{km}}\left(1-\frac{8mk}{\ln 2f(q)}\right) (1 - (k(m-1) + 2)\lambda) \\
    &\ge \frac{|E|^{k(m-1)+2}}{q^{km}}\left(1-\frac{8mk}{\ln 2f(q)} 
- 2km f(q)^{-\frac{2}{km}} \right).
\end{align}



\end{proof}

\begin{theorem}
    Assume that $|E| > q^{\frac{1}{2}(d+2-\frac{m-2}{m-1}+\delta)}$ for some $0 < \delta < \frac{1}{2m^2}$. Then,
    \[
    \mathcal{N}^*_{G_{k,m}}(E') \ge \frac{|E|^{k(m-1)+2}}{q^{km}} \left( 1 - \frac{2km}{\ln q} - \frac{4km}{\ln 2 q^{\frac{1}{m-1} + \delta}} - \frac{2(\ln q)^{(m-2)(k-2)-1} q^{d(k-2)+1}}{|E|^{k-1}} \right).
    \]
    The above is positive and tends to the statistically correct count when $q$ is sufficiently large and
    \[
    |E| > (\ln q)^{m-2} q^{\frac{d(k-2)+1}{k-1}}.
    \]
\end{theorem}
\begin{proof}
Suppose $\lambda \ge 4$ and let
\[
E' = \left\{x\in E: \text{ $\sum_{y\in E}S_t(x,y) \leq \lambda\frac{|E|}{q}$}\right\}.
\]
Then we have
\[
|E'|\geq \left(1-\frac{2}{\lambda}\right)|E|.
\]
The number of embeddings of a $2m$-cycle in $E'$ is at most 
\[
2\frac{|E|^{2m}}{q^{2m}}
\]
when $q$ is sufficiently large and 
\[
|E'|\geq q^{\frac{1}{2}(d+2-\frac{m-2}{m-1}+\delta')}
\]
where $0<\delta'<\frac{1}{2m^2}$ (simply take $\delta'$ marginally smaller than $\delta$ since $|E'| \ge \frac{1}{2}|E|$).
Since the ``branches" involve $(m-2)(k-3)+(m-3) = (m-2)(k-2)-1$ vertices (there are $k-3$ non-overlapping chains outside of the cycle, each having $m-2$ vertices on the branches, and there are $m-3$ vertices on the branches on the overlapping chain) and there are $k-2$ vertices remaining, we have that
\[
\mathcal{D}_{G_{k,m}}(E')\leq 2\frac{|E|^{2m}}{q^{2m}}\left(\frac{\lambda|E|}{q}\right)^{(m-2)(k-2)-1}\cdot q^{(d-2)(k-2)}
=2\lambda^{(m-2)(k-2)-1}\cdot \frac{|E|^{(m-2)k+3}}{q^{(m-d)k+2d-1}}.
\]
Now when $\lambda$ is sufficiently large
\[
\mathcal{N}_{G_{k,m}}(E')\geq 
\left(1-\frac{2}{\lambda}\right)^{k(m-1)+2}\cdot
\frac{|E|^{k(m-1)+2}}{q^{km}}\left(1-\frac{4mq^{\frac{d+1}{2}}}{\ln 2 \cdot |E'|}\right)^k.
\]
Then, suppose that $\lambda \to \infty$ as $q \to \infty$. The above can then be written as 
\[
\mathcal{N}_{G_{k,m}}(E')\geq \frac{|E|^{k(m-1)+2}}{q^{km}} \left( 1 - 2km\lambda^{-1} - \frac{4km}{\ln 2 q^{\frac{1}{m-1} + \delta}} \right)
\]
so that
\[
\mathcal{N}_{G_{k,m}}(E') - \mathcal{D}_{G_{k,m}}(E') \ge \frac{|E|^{k(m-1)+2}}{q^{km}} \left( 1 - 2km\lambda^{-1} - \frac{4km}{\ln 2 q^{\frac{1}{m-1} + \delta}} - 2\lambda^{(m-2)(k-2)-1}\frac{q^{d(k-2)+1}}{|E|^{k-1}} \right).
\]
Let $\lambda = \log q$ to get the result.

\end{proof}

The second bound is always tighter than the 
\[
|E|\gg q^{\frac{1}{2}(d+2-\frac{m-2}{m-1}+\delta)}
\]
required by the cycles result, except when $k=3$ or $(k,d,m) = (4,2,3), (4,2,4), (5,2,3)$. We summarize the bounds we obtain as follows.

\begin{theorem}\label{thm:summary_lower_bound}
    We have an asymptotic lower bound on $\mathcal{N}^*_{G_{k,m}}$ when
    \[
    |E| \ge \min\left\{ (\log q)\max\left( q^{\frac{d+1}{2}}, q^k \right), \max\left( q^{\frac{1}{2}(d+2-\frac{m-2}{m-1}+\delta)}, (\ln q)^{m-2} q^{\frac{d(k-2)+1}{k-1}} \right) \right\}.
    \]
\end{theorem}

When $k \ge 4$ and $(k,d,m) \neq (4,2,3), (4,2,4), (5,2,3)$, as discussed above, the condition in \Cref{thm:summary_lower_bound} becomes the following:
\[
|E| \geq \min \left\{ (\ln q)q^{\max\{k, \frac{d+1}{2}\} + \epsilon}, (\ln q)^{m-2}q^{\frac{d(k-2)+1}{k-1} + \epsilon}\right\}.
\]
In this case, the optimal bounds are as follows:
\begin{itemize}
    \item when $k < \frac{d+1}{2}$: $|E|\geq (\ln q)q^{\frac{d+1}{2}}$;
    \item when $\frac{d+1}{2}\leq k < d-1-\frac{1}{k-2}$: $|E|\ge (\ln q)q^{k}$;
    \item when $k\geq d-1-\frac{1}{k-2}$: $|E|\geq (\ln q)^{m-2}q^{\frac{d(k-2)+1}{k-1}}$.
\end{itemize}
For $k\geq 4$, $\frac{1}{k-2}<1$, so we have simplified ranges for the optimal bounds that do not involve $k$: 
\begin{itemize}
    \item when $k < \frac{d+1}{2}$: $|E|\geq (\ln q)q^{\frac{d+1}{2}}$;
    \item when $\frac{d+1}{2}\leq k \leq d-2$: $|E|\ge (\ln q)q^{k}$;
    \item when $k\geq d-1$: $|E|\geq (\ln q)^{m-2}q^{\frac{d(k-2)+1}{k-1}}$.
\end{itemize}

\section{Chains of $m$-Simplices}
\label{sec:chainsOfSimplices}
Throughout this section, let $m \geq 2$ and $k \geq 1$ be integers.

In this section, we prove a lower bound on the size of a subset $E$ such that it contains the statistically correct number of copies of a $k$-chain of $m$-simplices. We also show that under the same hypothesis, a \emph{non-degenerate} embedding exists.

Recall that $T^m_k$ denotes a $k$-chain of $m$-simplices, sometimes also referred to as a chain of $m$-simplices of length $k$. 
Note that the number of edges and vertices in $T_k^m$ are given by
\[
\#\text{edges} = \frac{(m+1)mk}{2}, \quad \#\text{vertices} = mk+1.
\]

\begin{proof}[Proof of \Cref{thm:kChainsTotal}]

Theorem 7 of \cite{IosevichParshall} implies the following asymptotic count of \emph{non-degenerate} embeddings of $T_1^m$ in a subset $E\subset\F_q^d$.
\begin{lemma}

    For $|E|\geq 12(m+1)^2 q^{m+\frac{d-1}{2}}$, we have
    \[
    \left|\mathcal{N}^*_{T_1^m}(E)-\frac{|E|^{m+1}}{q^{\binom{m+1}{2}}}\right|
    \leq 
    6(m+1)^2|E|^{m}q^{\frac{d-1}{2}-\frac{m(m-1)}{2}}.
    \]
    In particular, we have the following lower bound.
    \[
    \mathcal{N}^*_{T_1^m}(E)\geq
    \frac{|E|^{m+1}}{q^{\binom{m+1}{2}}}
    \left(1 - \frac{6(m+1)^2q^{\frac{d-1}{2}+m}}{|E|}\right).
    \]
    \label{lem:T_1^mNondegCount}
\end{lemma}
Note that this result is non-trivial only when $m<\frac{d+1}{2}$, so we assume that throughout.

Let $v$ be a degree $m$ vertex of the first $m$-simplex in $T_k^m$. For $x\in\F_q^d$, let
\[
f_{k,m}(x):=\#\{\text{embeddings of $T_k^m$ in $E$ with $v=x$}\}.
\]
We prove by induction that whenever
\[
|E|\geq 12(m+1)^2q^{m+\frac{d-1}{2}},
\]
the number of embeddings of chains of length $k = 2^\ell$ is lower bounded by
\[
\frac{|E|^{2^\ell m+1}}{q^{ 2^\ell \binom{m+1}{2}}}
\left(1 - \frac{2^{\ell+1}\cdot 3(m+1)^2q^{\frac{d-1}{2}+m}}{|E|}\right)
\]
using Cauchy-Schwarz. The base case $\ell=0$ was Lemma \ref{lem:T_1^mNondegCount}. The inductive step follows from
\begin{align*}
    \mathcal{N}_{T^m_{2^{\ell+1}}}(E)
&=\sum_{x\in E} f_{2^\ell,m}(x)^2 \\
&\geq 
|E|^{-1}\left(\sum_{x\in E} f_{2^\ell,m}(x)\right)^2\\
&\geq 
|E|^{-1}\mathcal{N}_{T^m_{2^{\ell}}}(E)^2\\
&\geq 
\frac{|E|^{2^{\ell+1} m+1}}{q^{ 2^{\ell+1} \binom{m+1}{2}}}
\left(1 - \frac{2^{\ell+2}\cdot 3(m+1)^2q^{\frac{d-1}{2}+m}}{|E|}\right).
\end{align*}
Now suppose $2^\ell < k\leq 2^{\ell+1}$.
We induct downwards $2^{\ell+1}-k$ steps to prove that whenever
\[
|E|\geq \left(1-\frac{2}{\lambda}\right)^{-1}12(m+1)^2 q^{m+\frac{d-1}{2}},
\]
we have 
\[
\mathcal{N}_{T^m_{k}}(E)\geq
\frac{|E|^{mk+1}}{q^{k\binom{m+1}{2}}}\cdot \frac{(1-\frac{2}{\lambda})^{2^{\ell+1}m+1}}{\lambda^{2^{\ell+1}-k}}
\left(1 - \frac{2^{\ell+2}\cdot 3(m+1)^2q^{\frac{d-1}{2}+m}}{(1-\frac{2}{\lambda})|E|}\right).
\]
The base case $k=2^{\ell+1}$ was done earlier. For the inductive step, observe that
\[
\mathcal{N}_{T^m_{k+1}}(E)=
\sum_{x\in E}f_{k,m}(x)f_{1,m}(x)
\leq 
\left(\max_{x\in E}f_{1,m}(x)\right)\sum_{x\in E}f_{k,m}(x)
=\left(\max_{x\in E}f_{1,m}(x)\right)\mathcal{N}_{T^m_{k}}(E).
\]
This implies that 
\[
\mathcal{N}_{T^m_{k}}(E)\geq \frac{\mathcal{N}_{T^m_{k+1}}(E)}{\max_{x\in E}f_{1,m}(x)}.
\]
Now consider
\[
E':=\left\{x\in E: f_{1,m}(x)\leq \lambda \cdot \frac{|E|^m}{q^{\binom{m+1}{2}}}\right\}.
\]
We have,
\[
|E\backslash E'|\leq\frac{q^{\binom{m+1}{2}}}{\lambda|E|^m}\sum_{x\in E}f_{1,m}(x)\leq \frac{2|E|}{\lambda},
\]
so 
\[
|E'|\geq |E|\left(1-\frac{2}{\lambda}\right)
\geq 12(m+1)^2 q^{m+\frac{d-1}{2}}.
\]
Then
\begin{align*} \label{gen_simplex_chain_lower_bound}
\mathcal{N}_{T_{k}}(E')
&\geq \mathcal{N}_{T^m_{k+1}}(E')\cdot{\frac{q^{\binom{m+1}{2}}}{\lambda|E|^m}}\\
&\geq
\mathcal{N}_{T_{2^{\ell+1}}^m}(E) \cdot \left(\frac{q^{\binom{m+1}{2}}}{\lambda|E|^m}\right)^{2^{\ell+1}-k}\\
&\geq
\frac{|E|^{mk+1}}{q^{k\binom{m+1}{2}}}\frac{(1-\frac{2}{\lambda})^{2^{\ell+1}m+1}}{\lambda^{2^{\ell+1}-k}}
\left(1 - \frac{2^{\ell+2}\cdot 3(m+1)^2q^{\frac{d-1}{2}+m}}{(1-\frac{2}{\lambda})|E|}\right).
\end{align*}
Setting $\lambda = 4$ yields the desired result.
\end{proof}
\begin{proof}[Proof of \Cref{thm:kChainsNondeg}]

Let
\[
E^* = \left\{x \in E: f_{1,m}(x) \le \lambda'\cdot \frac{|E|^m}{q^{\frac{m(m+1)}{2}}}\right\}.
\]
Then,
\[
|E \setminus E^*| \le \frac{q^{\frac{m(m+1)}{2}}}{\lambda' |E|^m} \sum_{x \in E}f_{1,m}(x) \le \frac{q^{\frac{m(m+1)}{2}}}{\lambda' |E|^m} \frac{2|E|^{m+1}}{q^{\frac{m(m+1)}{2}}} = \frac{2|E|}{\lambda'}.
\]
Let 
\[
E'' = \left\{x \in E: f_{1,m-1}(x) \le \lambda'\cdot \frac{|E|^{m-1}}{q^{\frac{m(m-1)}{2}}}\right\}.
\]
Then, similarly,
\[
|E \setminus E''| \le \frac{q^{\frac{m(m-1)}{2}}}{\lambda'\cdot  |E|^{m-1}} \sum_{x \in E}f_{1,m-1}(x) \le \frac{q^{\frac{m(m-1)}{2}}}{\lambda' |E|^{m-1}} \cdot \frac{2|E|^{m}}{q^{\frac{m(m-1)}{2}}} = \frac{2|E|}{\lambda'}.
\]
Thus, if we let $E_0 = E^* \cap E''$, then $|E \setminus E_0| \le 4|E|/\lambda'$. We count the number of degenerate embeddings of $T_k^m$ in $E_0$. We start with a point and multiply by $f_{1,m}$ to add a $m$-simplex. When we reach a repeated vertex we multiply by $f_{1,m-1}$. This yields an upper bound of
\[
|E| \left( \lambda'\frac{|E|^m}{q^{\frac{m(m+1)}{2}}} \right)^{k-1} \cdot \lambda'\cdot \frac{|E|^{m-1}}{q^{\frac{m(m-1)}{2}}} = (\lambda')^k\frac{|E|^{mk}}{q^{k\frac{m(m+1)}{2}-m}}.
\]
This is much less than
\[
c_{m,k}\cdot \frac{|E|^{mk+1}}{q^{k\frac{m(m+1)}{2}}}
\]
when $|E| \gg q^{m}$, which is always.
\end{proof}

\section{Trees of $m$-Simplices}
\label{sec:TreesOfSimplices}
In this section, we prove the lower bounds on the number of (non-degenerate) embeddings of trees of simplices given in \Cref{thm:TreesOfSimplicesTotal} and \Cref{thm:TreesOfSimplicesNondeg}. Throughout, assume $m\geq 2$. 

We prove the following proposition, where $\lambda$ and $\mu$ are parameters we will define in the proof.  
\begin{proposition}
\label{prop:countingTreesOfSimplices}
Let $T_\ell$ be a tree of $m$-simplices of size $\ell$. Suppose 
$|E|\geq 12(m+1)^2\mu^{1-\ell}q^{m+\frac{d-1}{2}}$. Then 
\[
    \mathcal{N}_{T_{\ell}}(E)\geq
    \frac{1}{2}\lambda^{1-\ell}(1-2\lambda^{-1})^{m\ell+1}\mu^{(\ell-1)(\ell m/2+1)}\frac{|E|^{m\ell+1}}{q^{\binom{m+1}{2}\cdot\ell}}.
\]
\end{proposition}
If we take $\lambda=4$, then $\mu=2^{-(m+5)}$, we obtain \Cref{thm:TreesOfSimplicesTotal}. Then using the same arguments from \Cref{sec:chainsOfSimplices} for counting degeneracies for chains of simplices, we obtain that the number of degenerate embeddings of $T_\ell$ is upper bounded by
\[
|E| \left( \lambda'\frac{|E|^m}{q^{\binom{m+1}{2}}} \right)^{\ell-1} \cdot \lambda'\frac{|E|^{m-1}}{q^{\frac{m(m-1)}{2}}} =
(\lambda')^\ell\frac{|E|^{m\ell}}{q^{\ell\binom{m+1}{2}-m}}.
\]
Upon taking $\lambda'=8$ and making the upper bound for degenerate embeddings half of the lower bound on $\mathcal{N}_{T_\ell}(E)$, we get \Cref{thm:TreesOfSimplicesNondeg}.

\begin{proof}[Proof of \Cref{prop:countingTreesOfSimplices}]

To count the number of embeddings of a tree of simplices, we simply induct on the size of the tree. The main estimate is the size of the following subset, $E^*$. Recall from the previous section that 
\[
E'=\left\{x\in E: f_{1,m}(x)\leq \lambda \cdot \frac{|E|^m}{q^{\binom{m+1}{2}}}\right\},
\]
where 
\[
f_{k,m}(x) =\#\{ \text{embeddings of $T_k^m$ into $E$ with $x$ as the base point}\}.
\]
Define \[E^* = \{x \in E': g_{1,m}(x) \ge \lambda^{-1}\frac{|E'|^m}{q^{\binom{m+1}{2}}}\},\] 
where $\lambda$ is to be determined later. Let 
\[
g_{k,m}(x) =\#\{ \text{embeddings of $T_k^m$ into $E'$ with $x$ as the base point}\}.
\]
Then,
\[
\frac{|E'|^{m+1}}{2q^{\binom{m+1}{2}}} 
\le \mathcal{N}_{T_1^m}(E')
= \sum_{x \in E'} g_{1,m}(x) \le |E'\setminus E^*| \lambda^{-1}\frac{|E'|^m}{q^{\binom{m+1}{2}}} + |E^*| \lambda \frac{|E|^m}{q^{\binom{m+1}{2}}}.
\]
\[
\implies |E^*| \ge \left( \frac{1}{2} - \frac{1}{\lambda}\right) \frac{|E'|^{m+1}}{q^{\binom{m+1}{2}}} \Big/ \left[ \lambda \frac{|E|^m}{q^{\binom{m+1}{2}}}\right] = \frac{1}{2\lambda}(1 - 2\lambda^{-1})^{m+2} |E|.
\]


Let 
\[
\mu:=\frac{(1-2\lambda^{-1})^{m+2}}{2\lambda}.
\]
Then the above becomes $|E^*|\geq \mu |E|$.
We prove that the number of embeddings of a tree containing $\ell$ connections into $E$ is at least
\[
    \frac{1}{2}\lambda^{1-\ell}(1-2\lambda^{-1})^{m\ell+1}\mu^{(\ell-1)(\ell m/2+1)}\frac{|E|^{m\ell+1}}{q^{\binom{m+1}{2}\cdot\ell}}
\]
when 
\[
|E|\geq 12(m+1)^2\mu^{1-\ell}q^{m+\frac{d-1}{2}}.
\]
The base case $\ell=1$ is done in \Cref{lem:T_1^mNondegCount}:
\[
\mathcal{N}_{T^m_1}(E)\geq \frac{|E|^{m+1}}{2q^{\binom{m+1}{2}}}.
\]
Let $T_\ell$ denote a tree of size $\ell$ of $m$-simplices. Fix an ending $m$-simplex in $T_\ell$ and let $T_{\ell-1}$ denote the tree by replacing that simplex in $T_\ell$ with one vertex. 
Now by our assumption,
\[
|E^*|\geq \mu|E|\geq 12(m+1)^2\mu^{1-(\ell-1)}q^{m+\frac{d-1}{2}},
\]
we can apply the induction hypothesis on $E^*$ to get
\begin{align*}
\mathcal{N}_{T_{\ell-1}}(E^*) &\geq 
\frac{1}{2}\lambda^{2-\ell}(1-2\lambda^{-1})^{m(\ell-1)+1}\mu^{(\ell-2)((\ell-1) m/2+1)}\frac{|E^*|^{m(\ell-1)+1}}{q^{\binom{m+1}{2}\cdot(\ell-1)}}\\
&\geq 
\frac{1}{2}\lambda^{2-\ell}(1-2\lambda^{-1})^{m(\ell-1)+1}\mu^{(\ell-1)(\ell m/2+1)}
\frac{|E|^{m(\ell-1)+1}}{q^{\binom{m+1}{2}\cdot(\ell-1)}}.
\end{align*}

Then we have
\begin{align*}
    \mathcal{N}_{T_\ell}(E)\geq
\mathcal{N}_{T_\ell}(E')
&\geq \mathcal{N}_{T_{\ell-1}}(E^*)\cdot\lambda^{-1}\frac{|E'|^m}{q^{\binom{m+1}{2}}}\\
&\geq 
\frac{1}{2}\lambda^{1-\ell}(1-2\lambda^{-1})^{m\ell+1}\mu^{(\ell-1)(\ell m/2+1)}\frac{|E|^{m\ell+1}}{q^{\binom{m+1}{2}\cdot\ell}},
\end{align*}
as desired.
\end{proof}

\subsection{H\"older Extension Method}
This is another method to count the number of embeddings of trees of $m$-simplices, but the constants are not explicitly determined.

Base case: Suppose that $T$ is a chain. We actually have that
\[
\mathcal{N}_{T}(E') \ge \frac{|E|^{mk+1}}{q^{k\binom{m+1}{2}}}\frac{(1-\frac{2}{\lambda})^{2km+1}}{\lambda^{k}}
\left(1 - \frac{4k\cdot 3(m+1)^2q^{\frac{d-1}{2}+m}}{(1-\frac{2}{\lambda})|E|}\right) \ge \frac{|E|^{mk+1}}{q^{k\binom{m+1}{2}}}\frac{(1-\frac{2}{\lambda})^{2km+1}}{2\lambda^{k}}
\]
provided that $|E|$ is large enough. Let $C_T = \frac{(1-\frac{2}{\lambda})^{2km+1}}{2\lambda^{k}}$ in this case.


At some node $v$, there are $r \ge 2$ chains attached to the rest of the tree. Let $T'$ be the tree obtained by making all of these chains the same length (the maximum length of the chains). Let this length be $\ell$ (i.e., $\ell$ simplices). Let $T_1$ be the tree obtained by eliminating $r-1$ of the chains from $T'$ and let $T_0$ be the tree obtain by eliminating all of them. Then, $T'$ is an $r$-H\"older extension of $T_1$, so
\[
\mathcal{N}_{T'} \ge \frac{\mathcal{N}_{T_1}^r}{\mathcal{N}_{T_0}^{r-1}}
\]
by Lemma \ref{holderLem}. Again define $E'$ so that $f_{1,m}(x) \le \lambda \frac{|E|^m}{q^{\binom{m+1}{2}}}$ for all $x \in E'$, then we have $|E \setminus E'| \le \frac{2|E|}{\lambda}$. Let $T_0$ have $v_0$ nodes (so it contains $v_0-1$ simplices). We then have that
\[
\mathcal{N}_{T_0} \le \lambda^{v_0-1} \frac{|E|^{mv_0-m+1}}{q^{(v_0-1)\binom{m+1}{2}}}
\]
and by the inductive hypothesis that
\[
\mathcal{N}_{T_1}(E') \geq C_{T_1}|E|\left(\frac{|E|^m}{q^{\binom{m+1}{2}}}\right)^{\ell+v_0-1},
\]
so that
\begin{align}
    \mathcal{N}_{T'}(E') &\ge C_{T_1}^r \left[|E|\left(\frac{|E|^m}{q^{\binom{m+1}{2}}}\right)^{\ell+v_0-1}\right]^r \cdot \left( \lambda^{v_0-1} \frac{|E|^{mv_0-m+1}}{q^{(v_0-1)\binom{m+1}{2}}} \right)^{1-r} \\
    &\ge C_{T_1}^r \frac{|E|^{rm(\ell+v_0-1) + r}}{q^{r(\ell+v_0-1)\binom{m+1}{2}}} \lambda^{(v_0-1)(1-r)} \frac{|E|^{(1-r)(mv_0-m+1)}}{q^{(1-r)(v_0-1)\binom{m+1}{2}}} \\
    &= C_{T_1}^r \lambda^{(v_0-1)(1-r)} \frac{|E|^{m(r\ell + v_0 - 1) + 1}}{q^{(r\ell + v_0 - 1)\binom{m+1}{2}}}.
\end{align}
In particular, we always have that $\mathcal{N}_{T'}(E') \ge C_{T'} \frac{|E|^{m(r\ell + v_0 - 1) + 1}}{q^{(r\ell + v_0 - 1)\binom{m+1}{2}}}$ for some constant $C_{T'}$ depending on $T'$. We can then use the fact that $\mathcal{N}_{T'}(E') \le \lambda \frac{|E|^m}{q^{\binom{m+1}{2}}}\mathcal{N}_{T'_{-1}}(E') $ where $T'_{-1}$ is $T'$ with one vertex removed. We use this fact repeatedly to obtain a lower bound on $\mathcal{N}_T(E')$. In particular, it is not hard to see that this must be done at most $(\ell-1)(r-1)$ times, so that we obtain the bound
\[
\mathcal{N}_T(E') \ge C_{T_1}^r \lambda^{(v_0+\ell-2)(1-r)} \frac{|E|^{mV-m+1}}{q^{E\binom{m+1}{2}}}.
\]
where $V$ and $E$ are the number of nodes and connections in $T$ respectively. This completes the induction.


\providecommand{\bysame}{\leavevmode\hbox to3em{\hrulefill}\thinspace}
\providecommand{\MR}{\relax\ifhmode\unskip\space\fi MR }
\providecommand{\MRhref}[2]{%
  \href{http://www.ams.org/mathscinet-getitem?mr=#1}{#2}
}
\providecommand{\href}[2]{#2}

\printbibliography

\end{document}